\documentclass[12pt,a4paper,reqno]{amsart}
\usepackage[pdftex]{graphicx}
\usepackage{hyperref}
\usepackage{cleveref}
\usepackage{amscd}
\usepackage{amsmath}
\usepackage{amsthm}
\usepackage{amsfonts}
\usepackage{amssymb}
\usepackage{verbatim}
\usepackage{enumerate}
\usepackage{siunitx}
\usepackage{bm}
\usepackage{wrapfig}
\usepackage{mathtools}
\usepackage[tableposition=top]{caption}
\usepackage{booktabs,dcolumn}
\usepackage{tikz}
\numberwithin{equation}{section}

\theoremstyle{plain}

\newtheorem{theorem}{Theorem}[section]
\newtheorem{proposition}[theorem]{Proposition}

\newtheorem{corollary}[theorem]{Corollary}

\theoremstyle{definition}
\newtheorem{definition}[theorem]{Definition}
\newtheorem{remark}[theorem]{Remark}

\newtheorem*{acknowledgement}{Acknowledgements}

\begin{document}

\title[Inequalities on six points in a $\mathrm{CAT}(0)$ space]{Inequalities on six points in a $\mathrm{CAT}(0)$ space}

\author{Tetsu Toyoda}
\address{Kogakuin University, 2665-1, Nakano, Hachioji, Tokyo, 192-0015 Japan}
\email{toyoda@cc.kogakuin.ac.jp}
%\date{\today}
\subjclass[2020]{30L15, 53C23, 51F99}
\keywords{inequalities in metric spaces, $\mathrm{CAT}$(0) spaces, isometric embeddability}
%\thanks{This work was supported in part by JSPS KAKENHI Grant Number JP21K03254.}
%\maketitle

%%%%%%%%%%%%%%%%%%%%%%%%%%%%%%%%%%%%%%%%%%%%%%%%%

\begin{abstract}
We establish a family of inequalities that hold true on any $6$ points in any $\mathrm{CAT}(0)$ space. 
We prove that the validity of these inequalities does not follow from any properties of $5$-point subsets of $\mathrm{CAT}(0)$ spaces. 
In particular, the validity of these inequalities does not follow from the $\mathrm{CAT}(0)$ $4$-point condition. 
\end{abstract}

\maketitle

\section{Introduction}

Many simple necessary and sufficient conditions for a geodesic metric space 
to be a $\mathrm{CAT}(0)$ space have been known. 
However, to find a characterization of those metric spaces that admit a distance-preserving embedding into a $\mathrm{CAT}(0)$ space is 
an open problem posed by Gromov 
(see \cite[Section 1.19+]{Gr1}, \cite[\S15]{Gr2}, \cite[Section 7]{AKP}, \cite[Chapter 10, E]{AKP-book}, \cite[Section 1.4]{ANN}, 
\cite[Section 2.2]{Bac} and \cite[Section 1]{toyoda-five}). 

A condition for points $x_1 ,\ldots ,x_n$ in a metric space $(X,d_X )$ is called 
a {\em quadratic metric inequality on $n$ points} 
if it can be written in the form 
\begin{equation}\label{qm-ineq}
0\leq
\sum_{i=1}^n\sum_{j=1}^n a_{ij}d_X (x_i ,x_j)^2
\end{equation}
by a real symmetric $n\times n$ matrix $(a_{ij})$ with $a_{kk}=0$ for any $k\in\{ 1,\ldots ,n\}$ (see \cite[Section 1.4.1]{ANN}). 
If \eqref{qm-ineq} holds true for all $x_1 ,\ldots ,x_n \in X$, then 
we say that the metric space $X$ {\em satisfies the quadratic metric inequality} \eqref{qm-ineq}. 
We call a quadratic metric inequality a {\em $\mathrm{CAT}(0)$ quadratic metric inequality} if 
every $\mathrm{CAT}(0)$ space satisfies it. 
Andoni, Naor and Neiman \cite[Proposition~3]{ANN} proved that 
an $n$-point metric space admits a distance-preserving embedding into a $\mathrm{CAT}(0)$ space if and only if it 
satisfies all $\mathrm{CAT}(0)$ quadratic metric inequalities on $n$ points. 
This result tells us that we can characterize those metric spaces that admit a distance-preserving embedding into a $\mathrm{CAT}(0)$ space 
if we can characterize $\mathrm{CAT}(0)$ quadratic metric inequalities.

We recall the following family of $\mathrm{CAT}(0)$ quadratic metric inequalities on $4$ points. 

\begin{definition}
We say that a metric space $(X,d_X )$ satisfies 
the {\em $\boxtimes$-inequalities} if we have 
\begin{multline*}
0
\leq
(1-s)(1-t) d_X (x,y)^2 +s(1-t) d_X (y,z)^2 +st d_X (z,w)^2 +(1-s)t d_X (w,x)^2 \\
-s(1-s) d_X (x,z)^2 -t(1-t) d_X (y,w)^2
\end{multline*}
for any $s, t \in\lbrack0,1\rbrack$ and any $x,y,z,w\in X$.  
\end{definition}
Gromov \cite{Gr2} and Sturm \cite{St} proved independently that every $\mathrm{CAT}(0)$ space satisfies the 
$\boxtimes$-inequalities. 
The name ``$\boxtimes$-inequalities" is based on a notation used by Gromov \cite{Gr2}. 
Sturm \cite{St} called these inequalities the {\em weighted quadruple inequalities}. 
The present author \cite{toyoda-five} proved that 
if a metric space $X$ satisfies the $\boxtimes$-inequalities, 
then $X$ satisfies all $\mathrm{CAT}(0)$ quadratic metric inequalities on $5$ points, which implies 
the following theorem. 

\begin{theorem}[\cite{toyoda-five}]\label{five-th}
A metric space $X$ with $|X|\leq 5$ admits 
a distance-preserving embedding into a $\mathrm{CAT}(0)$ space if and only if $X$ satisfies the 
$\boxtimes$-inequalities. 
\end{theorem}

On the other hand, Nina Lebedeva constructed a family of $6$-point metric spaces that satisfy 
the $\boxtimes$-inequalities but do not admit a distance-preserving embedding into any $\mathrm{CAT}(0)$ space 
(see the arXiv version of \cite[Section 7.2]{AKP}). 
Therefore, 
it follows from the result of Andoni, Naor and Neiman that 
there exist $\mathrm{CAT}(0)$ quadratic metric inequalities on $6$ points that 
do not follow from the $\boxtimes$-inequalities. 
However, to the best of the author's knowledge, 
no such $\mathrm{CAT}(0)$ quadratic metric inequality has ever been known explicitly. 
In this paper, we establish the first explicit examples of such inequalities. 
The following theorem is our main result.

\begin{theorem}\label{6-pt-ineq-th}
Let $(X,d_{X})$ be a $\mathrm{CAT}(0)$ space. 
For any $a,b,c,s,t\in\lbrack 0,1\rbrack$ with $a\leq s$, 
and any points $x_0 ,x_1 ,y_0 ,y_1 ,z_0 ,z_1\in X$, we have
\begin{align}
&sta\big( (1-t)(1-a)+(1-s)tb\big) \, d_{X}(x_0 ,x_1 )^2 \nonumber\\
&\hspace{4pt}+s(1-t)t(1-b)b \, d_{X}(y_0 ,y_1 )^2 
+ab(1-c)c \, d_{X}(z_0 ,z_1 )^2 \nonumber\\
&\hspace{8pt}\leq%\,
(1-s)tab(1-c)\, d_{X}(x_0 ,z_0 )^2
+stab(1-c)\, d_{X}(x_1 ,z_0 )^2 
+(1-t)ab(1-c) \, d_{X}(y_1 ,z_0 )^2 \nonumber\\
&\hspace{12pt}+(1-s)tabc \, d_{X}(x_0 ,z_1 )^2 
+stabc \, d_{X}(x_1 ,z_1 )^2 
+(1-t)abc\, d_{X}(y_1 ,z_1 )^2 \nonumber\\
&\hspace{16pt}+s(1-t)t(1-a)(1-b) \, d_{X}(x_0 ,y_0 )^2
+s(1-t)ta(1-b) \, d_{X}(x_1 ,y_0 )^2 \nonumber\\ 
&\hspace{20pt}+(1-t)t(s-a)b \, d_{X}(x_0 ,y_1 )^2 .
\label{6-pt-ineq}
\end{align}
Moreover, for any $a,b,c,s,t\in(0,1)$ with $a< s$, there exists a $6$-point metric space $(L,d_{L})$ 
that satisfies the following two conditions:
\begin{enumerate}
\item[$\mathrm{(i)}$]
Every $5$-point subset of $L$ admits a distance-preserving embedding into a $\mathrm{CAT}(0)$ space. 
\item[$\mathrm{(ii)}$]
$L$ does not satisfy the inequality \eqref{6-pt-ineq}
for the parameters $a,b,c,s,t$. 
More precisely, there exist $x_0 ,x_1 ,y_0 ,y_1 ,z_0 ,z_1 \in L$ such that 
\eqref{6-pt-ineq} does not hold true for $d_{X}=d_{L}$ and the given parameters $a,b,c,s,t$. 
\end{enumerate}
\end{theorem}

In the proof of the moreover part of Theorem~\ref{6-pt-ineq-th}, we prove that 
we can find the desired $6$-point metric space $L$ in Lebedeva's family of $6$-point metric spaces.

\subsection{The Andoni-Naor-Neiman inequalities}

Andoni, Naor and Neiman \cite{ANN} established the following family of $\mathrm{CAT}(0)$ quadratic metric inequalities (see Lemma~27 and Section 5.1 in \cite{ANN}). 
We denote by $\lbrack n\rbrack$ the set $\{ 1,2,\ldots ,n\}$ for each positive integer $n$. 

\begin{theorem}[Andoni, Naor and Neiman \cite{ANN}]\label{ANN-ineq-th}
Fix positive integers $m$ and $n$. 
Let $c_1 ,\ldots ,c_m$ be positive real numbers. 
Suppose for every $k\in\lbrack m\rbrack$, $p_1^k ,\ldots ,p_n^k ,q_1^k ,\ldots ,q_n^k$ are positive real numbers that satisfy 
\begin{equation*}
\sum_{i=1}^n p_i^k =\sum_{j=1}^n q_j^k =1.
\end{equation*}
Suppose for every $k\in\lbrack m\rbrack$, 
$A_k=(a_{ij}^k)$ and $B_k=(b_{ij}^k)$ are $n\times n$ matrices with nonnegative real entries 
that satisfy 
\begin{equation*}
\sum_{s=1}^n a_{is}^{k}+\sum_{s=1}^n b_{sj}^{k}=p_i^k +q_j^k
\end{equation*}
for any $i,j\in\lbrack n\rbrack$. 
Then we have 
\begin{equation}\label{ANN-ineq}
\sum_{k=1}^{m}\hspace{2mm}\sum_{i,j\in\lbrack n\rbrack : a_{ij}^k +b_{ij}^k >0}\frac{c_k a_{ij}^k b_{ij}^k}{a_{ij}^k +b_{ij}^k}d_X (x_i ,x_j )^2
\leq
\sum_{k=1}^{m}\sum_{i=1}^n \sum_{j=1}^n c_k p_i^k q_j^k d_X (x_i ,x_j )^2
\end{equation}
for any $\mathrm{CAT}(0)$ space $(X,d_X)$ and any points $x_1 ,\ldots ,x_n \in X$. 
\end{theorem}

It is not hard to see that 
the $\boxtimes$-inequalities can be written in the form \eqref{ANN-ineq}. 
Moreover, as stated in \cite[Section 5.1]{ANN}, 
it seems that all of the previously used $\mathrm{CAT}(0)$ quadratic metric inequalities, 
are of the form \eqref{ANN-ineq}.
Recently, 
the present author \cite{toyoda-ANN} proved that 
Lebedeva's $6$-point metric spaces 
satisfy all inequalities of the form \eqref{ANN-ineq} (see Theorem~\ref{L-ANN-th}). 
Together with this fact, the argument in the proof of Theorem~\ref{6-pt-ineq-th} yields the following corollary. 

\begin{corollary}\label{6-pt-ANN-coro}
For any $a,b,c,s,t\in(0,1)$ with $a<s$, there exists a $6$-point metric space $L$ 
that satisfies the following two conditions:
\begin{enumerate}
\item[$(\mathrm{i})$]
$L$ satisfies the inequality \eqref{ANN-ineq} for any $m$, $n$, 
$c_1 ,\ldots ,c_m$, $\left( p_{i}^{1}\right)_{i\in\lbrack n\rbrack},\ldots ,\left( p_{i}^{m}\right)_{i\in\lbrack n\rbrack}$, 
$\left( q_{j}^{1}\right)_{j\in\lbrack n\rbrack},\ldots ,\left( q_{j}^{m}\right)_{j\in\lbrack n\rbrack}$, 
$A_1 ,\ldots ,A_m$, $B_1 ,\ldots ,B_m$ as in the statement of Theorem~\ref{ANN-ineq-th}.
\item[$(\mathrm{ii})$]
$L$ does not satisfy the inequality \eqref{6-pt-ineq}
for the parameters $a,b,c,s,t$. 
\end{enumerate}
\end{corollary}

We will prove Corollary~\ref{6-pt-ANN-coro} in Section \ref{proof-sec}.

\subsection{Graph comparison}

Gromov \cite[\S 7]{Gr2} introduced the following notion. 

\begin{definition}[Gromov \cite{Gr2}]\label{Cycl-def}
Let $n$ be an integer at least $4$, and let $C_{n}=(V,E)$ be the $n$-vertex cycle graph with vertex set $V$ and edge set $E$. 
A metric space $(X,d_X )$ is said to satisfy the {\em $\mathrm{Cycl}_{n}(0)$ condition} if 
for any map $f\colon V\to X$, there exists a map $g\colon V\to\mathcal{H}$ to a Hilbert space $\mathcal{H}$ such that 
\begin{equation*}
\begin{cases}
\|g(u)-g(v)\|\leq d_X (f(u),f(v)),\quad\textrm{if }\{ u,v\}\in E,\\
\|g(u)-g(v)\|\geq d_X (f(u),f(v)),\quad\textrm{if }\{ u,v\}\not\in E
\end{cases}
\end{equation*}
for any $u,v\in V$. 
\end{definition}

By Reshetnyak's majorization theorem \cite{R}, every $\mathrm{CAT}(0)$ space satisfies 
the $\mathrm{Cycl}_n (0)$ condition for any integer $n\geq 4$. 
Gromov \cite{Gr2} proved that a geodesic metric space is $\mathrm{CAT}(0)$ if and only if 
it satisfies the $\mathrm{Cycl}_4 (0)$ condition. 
He also proved that a metric space satisfies the $\mathrm{Cycl}_4 (0)$ condition if and only if 
it satisfies the $\boxtimes$-inequalities. 
Thus, by Theorem~\ref{five-th}, a metric space $X$ with $|X|\leq 5$ admits a distance-preserving embedding into a $\mathrm{CAT}(0)$ space if and only if 
$X$ satisfies the $\mathrm{Cycl}_{4}(0)$ condition. 
Furthermore, the present author \cite{toyoda-cycle} proved that 
if a metric space satisfies the $\mathrm{Cycl}_{4}(0)$ condition, then it also satisfies the $\mathrm{Cycl}_{n}(0)$ condition for 
any integer $n\geq 4$. 
On the other hand, the existence of Lebedeva's $6$-point metric spaces that we have mentioned before shows that 
the $\mathrm{Cycl}_4 (0)$ condition is not a sufficient condition for a $6$-point metric space to admit a distance-preserving embedding into a $\mathrm{CAT}(0)$ space. 

Lebedeva, Petrunin and Zolotov \cite{LPZ} introduced the following more general notion. 
We assume that graphs are always simple and undirected. 

\begin{definition}[Lebedeva, Petrunin and Zolotov \cite{LPZ}]\label{G-comp}
Let $G=(V,E)$ be a graph with vertex set $V$ and edge set $E$. 
A metric space $(X,d_X )$ is said to satisfy the {\em $G$-comparison} if 
for any map $f\colon V\to X$, there exists a map $g\colon V\to\mathcal{H}$ to a Hilbert space $\mathcal{H}$ such that 
\begin{equation*}
\begin{cases}
\|g(u)-g(v)\|\leq d_X (f(u),f(v)),\quad\textrm{if }\{ u,v\}\in E,\\
\|g(u)-g(v)\|\geq d_X (f(u),f(v)),\quad\textrm{if }\{ u,v\}\not\in E
\end{cases}
\end{equation*}
for any $u,v\in V$.  
\end{definition}

For each integer $n\geq 4$, the $C_n$-comparison is no other than the $\mathrm{Cycl}_{n}(0)$ condition. 
For recent developments concerning graph comparison, see \cite{LP-tree}, \cite{LP-graph} and \cite{LPZ}. 
Let $O_3$ be the octahedron graph (see \textsc{Figure}~\ref{O3-fig}). 
Lebedeva and Petrunin \cite{LP-to} posed the question whether 
all $\mathrm{CAT}(0)$ spaces satisfy the $O_3$-comparison, and if the answer is yes, 
whether the $O_3$-comparison becomes a sufficient condition for a $6$-point metric space to admit a distance-preserving embedding into a $\mathrm{CAT}(0)$ space. 
In \cite{LP-tree}, Lebedeva and Petrunin proved that products of metric trees satisfy the $O_3$-comparison. 
However, it remains an open question whether all $\mathrm{CAT}(0)$ spaces satisfy the $O_3$-comparison. 
We note that the validity of the inequalities \eqref{6-pt-ineq} follows from the $O_3$-comparison.

\begin{figure}[htbp]
\centering
\begin{tikzpicture}[scale=0.4]
\draw (0,0) -- (1,1.73);
\draw (1,1.73) -- (3,1.73);
\draw (3,1.73) -- (4,0);
\draw (4,0) -- (3,-1.73);
\draw (3,-1.73) -- (1,-1.73);
\draw (1,-1.73) -- (0,0);
\draw (0,0) -- (3,1.73);
\draw (1,-1.73) -- (4,0);
\draw (0,0) -- (3,-1.73);
\draw (1,1.73) -- (4,0);
\draw (1,1.73) -- (1,-1.73);
\draw (3,1.73) -- (3,-1.73);
\draw [fill] (0,0) circle [radius=0.15];
\draw [fill] (1,1.73) circle [radius=0.15];
\draw [fill] (3,1.73) circle [radius=0.15];
\draw [fill] (4,0) circle [radius=0.15];
\draw [fill] (1,-1.73) circle [radius=0.15];
\draw [fill] (3,-1.73) circle [radius=0.15];
\end{tikzpicture}
\caption{The octahedron graph $O_3$. }
\label{O3-fig}
\end{figure}

\begin{proposition}\label{O3-prop}
If a metric space $(X,d_{X})$ satisfies the $O_3$-comparison, then $X$ satisfies the inequality \eqref{6-pt-ineq} for 
any $a,b,c,s,t\in\lbrack 0,1\rbrack$ with $a\leq s$. 
\end{proposition}

Proposition~\ref{O3-prop} follows from Theorem \ref{6-pt-ineq-th} and the definition of 
the $O_3$-comparison immediately. 
We will prove it at the end of Section \ref{proof-sec}.

\subsection{Towards a characterization of $6$-point subsets of $\mathrm{CAT}(0)$ spaces}
It is natural to ask whether the validity of all inequalities of the form \eqref{6-pt-ineq} is a sufficient condition for 
a $6$-point metric space to admit a distance-preserving embedding into a $\mathrm{CAT}(0)$ space. 
However, at this moment, the author has no guess whether the answer is yes or no.

\subsection{Organization of the paper}
This paper is organized as follows. 
In Section \ref{preliminaries-sec}, we recall some facts concerning $\mathrm{CAT}(0)$ spaces, and recall the definition of Lebedeva's $6$-point metric spaces. 
In Section \ref{proof-sec}, we prove Theorem~\ref{6-pt-ineq-th}, Corollary~\ref{6-pt-ANN-coro} and Proposition~\ref{O3-prop}.

\section{Preliminaries}\label{preliminaries-sec}

In this section, we recall some facts concerning $\mathrm{CAT}(0)$ spaces, and recall the definition of Lebedeva's $6$-point metric spaces. 
For a detailed exposition of $\mathrm{CAT}(0)$ spaces, see \cite{AKP-book}, \cite{BH}, \cite{BBI} and \cite{St}. 

A {\em geodesic} in a metric space $X$ is a distance-preserving embedding of an interval of the real line into $X$. 
For any $x,y\in X$, the image of a geodesic $\gamma\colon \lbrack 0,d_X (x,y)\rbrack\to X$ with $\gamma (0)=x$ and $\gamma (d_{X}(x,y))=y$ is 
called a {\em geodesic segment with endpoints $x$ and $y$}. 
A metric space $X$ is called {\em geodesic} if for any $x,y\in X$, there exists a geodesic segment with endpoints $x$ and $y$. 
We recall the definition of $\mathrm{CAT}(0)$ spaces. 

\begin{definition}\label{CAT(0)-def}
A complete geodesic metric space $(X, d_X )$ is called a {\em $\mathrm{CAT}(0)$ space} if 
for any $x,y\in X$, and any geodesic $\gamma\colon\lbrack 0,d_X (x,y)\rbrack\to X$ 
with $\gamma(0) = x$ and $\gamma(d_X (x,y)) = y$, 
the inequality 
\begin{align}\label{CAT(0)-def-ineq}
d_X \left(z, \gamma(t d_X (x,y) )\right)^2 \leq 
 (1-t)d_X (z,x) ^2 +td_X (z,y)^2 -(1-t)t d_X (x,y)^2
\end{align}
holds true for any $z\in X$ and any $t\in\lbrack 0,1\rbrack$. 
\end{definition}

In a $\mathrm{CAT}(0)$ space $X$, for any $x,y\in X$, a geodesic segment with endpoints $x$ and $y$ exists uniquely. 
It is straightforward to see that \eqref{CAT(0)-def-ineq} holds with equality when $X$ is a Hilbert space. 
In particular, Hilbert spaces are $\mathrm{CAT}(0)$. 

\begin{remark}
Let $(X,d_X )$ be a $\mathrm{CAT}(0)$ space, and let $x,y,z\in X$. 
Suppose that $\gamma\colon\lbrack 0,d_X (x,y)\rbrack\to X$ is the geodesic such that $\gamma(0) = x$ and $\gamma(d_X (x,y)) = y$. 
Then Definition~\ref{CAT(0)-def} implies that 
\begin{equation}\label{quadratic-tri-ineq}
0\leq (1-t)d_X (z,x) ^2 +td_X (z,y)^2 -(1-t)t d_X (x,y)^2
\end{equation}
for any $t\in\lbrack 0,1\rbrack$, and 
equality holds in \eqref{quadratic-tri-ineq} if and only if $z=\gamma (td_X (x,y))$. 
Although not necessary for our purposes, we note that \eqref{quadratic-tri-ineq} actually follows only from 
the triangle inequality 
\begin{equation*}
d_{X}(x,y)\leq d_{X}(z,x)+d_{X}(z,y)
\end{equation*}
and the nonnegativity of $d_{X}$. 
Therefore, \eqref{quadratic-tri-ineq} holds in every metric space. 
\end{remark}

We recall the notion of barycenters. 
Barycenters in this context were considered by Kleiner \cite{Kl} in his work about the dimension of 
Alexandrov spaces with curvature bounded above. 

Let $(X,d_X )$ be a $\mathrm{CAT}(0)$ space, and let $x_1 ,\ldots ,x_n \in X$. 
Suppose $p_1 ,\ldots ,p_n$ are positive real numbers with $\sum_{i=1}^n p_i =1$. 
Then there exists a unique point $z\in X$ that satisfies 
\begin{equation}\label{variance-ineq}
d_X (w ,z)^2 +\sum_{i=1}^n p_i d_X (z, x_i )^2
\leq
\sum_{i=1}^n p_i d_X (w, x_i )^2
\end{equation}
for any $w\in X$ (see Lemma~4.4 and Theorem~4.9 in \cite{St}). 
The point $z\in X$ with this property is called the {\em barycenter} of the probability measure $\mu =\sum_{i=1}^n p_i \delta_{x_i}$ on $X$, 
where 
$\delta_{x_i}$ is the Dirac measure at $x_i$. 
We denote the barycenter of $\mu$ by $\mathrm{bar}(\mu)$. 
We note that for any geodesic 
$\gamma\colon\lbrack 0,d_X (x,y)\rbrack\to X$ with $\gamma(0) = x$ and $\gamma(d_X (x,y)) = y$, 
we have 
\begin{equation*}
\mathrm{bar}((1-t)\delta_{x}+t\delta_{y})=\gamma (td_X (x,y))
\end{equation*}
for any $t\in\lbrack 0,1\rbrack$. 
For any finitely supported probability measures $\mu =\sum_{i=1}^{n} p_i \delta_{x_i}$ and $\nu =\sum_{j=1}^{m} q_j \delta_{y_j}$ on $X$, 
by applying \eqref{variance-ineq} twice, we obtain the following inequality (see \cite[Section 5.2]{ANN}): 
\begin{multline}\label{double-variance-ineq}
d_{X}(\mathrm{bar}(\mu ),\mathrm{bar}(\nu ))^2
+\sum_{i=1}^{n} p_i d_X (\mathrm{bar}(\mu ), x_i )^2
+\sum_{j=1}^{m} q_j d_X (\mathrm{bar}(\nu ), y_j )^2 \\
\leq
\sum_{i=1}^{n}\sum_{j=1}^{m} p_{i}q_{j} d_X (x_i ,y_j )^2 .
\end{multline}
When $X$ is a Hilbert space, 
we have $\mathrm{bar}(\mu )=\sum_{i=1}^{n}p_{i}x_{i}$ for any finitely supported probability measure 
$\mu =\sum_{i=1}^{n} p_i \delta_{x_i}$ on $X$. 
Therefore, it follows from direct computations that 
both \eqref{variance-ineq} and \eqref{double-variance-ineq} hold with equality when $X$ is a Hilbert space.

Although the following proposition is an immediate consequence of Definition~\ref{CAT(0)-def}, 
it plays a key role in the proof of Theorem~\ref{6-pt-ineq-th}. 

\begin{proposition}\label{thick-prop}
Let $(X,d_{X})$ be a $\mathrm{CAT}(0)$ space, and fix $x_0 ,x_1 ,y\in X$. 
Suppose that $\gamma\colon \lbrack 0,d_{X}(x_{0},x_{1})\rbrack\to X$ is a geodesic such that 
$\gamma (0)=x_0$ and $\gamma (d_{X}(x_0 ,x_{1}))=x_1$. 
Set $x_{t}=\gamma (t d_{X}(x_0 ,x_{1}))$ for each $t\in\lbrack 0,1\rbrack$. 
Fix $a,s\in\lbrack 0,1\rbrack$ with $a\leq s$. 
Then we have 
\begin{equation*}
a d_{X}(y, x_{s})^2
\geq
s d_{X}(y,x_{a})^2
-(s-a)d_{X}(y,x_{0})^2
+sa(s-a)d_{X}(x_{0},x_{1})^2 ,
\end{equation*}
and equality holds when $X$ is a Hilbert space. 
\end{proposition}

\begin{proof}
If $s=0$, then the desired inequality holds with equality clearly. 
If $s>0$, 
then by Definition~\ref{CAT(0)-def}, we have 
\begin{align*}
d_{X}(y,x_{a})^2
&\leq
\frac{s-a}{s}d_{X}(y,x_{0})^2 +\frac{a}{s}d_{X}(y,x_{s})^2 -\frac{(s-a)a}{s^2}d_{X}(x_{0},x_{s})^2 \\
&=
\frac{s-a}{s}d_{X}(y,x_{0})^2 +\frac{a}{s}d_{X}(y,x_{s})^2 -(s-a)a d_{X}(x_{0},x_{1})^2 ,
\end{align*}
and equality holds when $X$ is a Hilbert space. 
Thus we have 
\begin{equation*}
a d_{X}(y, x_{s})^2
\geq
s d_{X}(y,x_{a})^2
-(s-a)d_{X}(y,x_{0})^2
+sa(s-a)d_{X}(x_{0},x_{1})^2 ,
\end{equation*}
and equality holds when $X$ is a Hilbert space. 
\end{proof}

We now recall the $6$-point metric space constructed by Nina Lebedeva, 
which appeared in the arXiv version of \cite[Section 7.2]{AKP} (not in the published version). 
For any $p,q\in\mathbb{R}^3$, 
we denote the line segment $\{ (1-t)p+tq\, |\,t\in\lbrack 0,1\rbrack\}$ by $\lbrack p,q\rbrack$, 
and the set $\lbrack p,q\rbrack\setminus\{ p,q\}$ by $(p,q)$. 
For any subset $S$ of $\mathbb{R}^3$, we denote by $\mathrm{conv}(S)$ the convexhull of $S$ in the Euclidean space $\mathbb{R}^3$. 
Suppose $x_0 ,x_1 ,y_0 ,y_1 ,z_0 ,z_1$ are distinct six points in 
$\mathbb{R}^3$ such that 
\begin{equation}\label{lebedeva-condition}
|(x_0 ,x_1 )\cap (y_0, y_1 )| =1,\quad
\left|(z_0 ,z_1 )\cap\left(\mathrm{conv}(\{ x_0 ,y_0 ,x_1 ,y_1 \})\setminus\mathcal{I}\right)\right| =1,
\end{equation}
where 
$\mathcal{I}=\lbrack x_{0},x_{1}\rbrack\cup\lbrack y_{0},y_{1}\rbrack
\cup\lbrack x_{0},y_{0}\rbrack\cup\lbrack y_{0},x_{1}\rbrack
\cup\lbrack x_{1},y_{1}\rbrack\cup\lbrack y_{1},x_{0}\rbrack$. 
It follows from \eqref{lebedeva-condition} that 
the points $x_{0},y_{0},x_{1},y_{1}$ form the vertices of a convex quadrilateral in a $2$-dimensional affine subspace of $\mathbb{R}^3$, 
and that the points $x_{0},x_{1},y_{0},y_{1},z_{0},z_{1}$ form the vertices of a nonregular convex octahedron in $\mathbb{R}^3$
(see \textsc{Figure}~\ref{lebedeva-fig}). 
We set $L=\{ x_{0},x_{1},y_{0},y_{1},z_{0},z_{1}\}$. 
For any $\varepsilon\geq 0$, 
we define a map $d_{\varepsilon}\colon L\times L\to\lbrack 0,\infty )$ by 
\begin{equation}\label{L-distance-def}
d_{\varepsilon} (p,q)
=
\begin{cases}
\|p-q\| +\varepsilon ,\quad\textrm{if }\{ p,q\}=\{ z_0 ,z_1 \} ,\\
\|p-q\| ,\quad\textrm{if }\{ p,q\}\neq\{ z_0 ,z_1 \}
\end{cases}
\end{equation}
for any $p,q\in L$,
where $\| p-q\|$ is the Euclidean norm of $p-q$.

\begin{figure}[htbp]
\centering
\begin{tikzpicture}[scale=0.6]
\draw [opacity=0.2,fill=gray] (0,0) -- (3,5) -- (14,5) -- (11,0) -- (0,0);
\draw [opacity=0.3,fill=gray] (3,1) -- (10,1) -- (11,4) -- (5,4) -- (3,1);
\draw (3,1) -- (11,4);
\draw (10,1) -- (5,4);
\draw (9,6) -- (7.9,1.6);
%%%%%%%%%
\draw (3,1) -- (10,1);
\draw (10,1) -- (11,4);
\draw (11,4) -- (5,4);
\draw (5,4) -- (3,1);
%%%%%%%%%
\draw (7.9,1.6)[dotted] -- (7.5,0);
\draw (7.5,0) -- (7,-2);
\draw [fill] (3,1) circle [radius=0.07];
\draw [fill] (11,4) circle [radius=0.07];
\draw [fill] (9,6) circle [radius=0.07];
\draw [fill] (7.9,1.6) circle [radius=0.07];
\draw [fill] (7,-2) circle [radius=0.07];
\draw [fill] (10,1) circle [radius=0.07];
\draw [fill] (5,4) circle [radius=0.07];
\node [right] at (9,6) {$z_1$};
\node [right] at (7,-2) {$z_0$};
\node [right] at (11,4) {$x_0$};
\node [left] at (3,1) {$x_1$};
\node [right] at (10,1) {$y_1$};
\node [left] at (5,4) {$y_0$};
\end{tikzpicture}
\caption{Lebedeva's six points in $\mathbb{R}^3$. }
\label{lebedeva-fig}
\end{figure}

\begin{theorem}[Lebedeva]\label{Lebedeva-th}
Suppose $x_{0},x_{1},y_{0},y_{1},z_{0},z_{1}\in\mathbb{R}^3$ 
are distinct six points that satisfy \eqref{lebedeva-condition}. 
Set $L=\{ x_{0},x_{1},y_{0},y_{1},z_{0},z_{1}\}$. 
Then there exists $C\in (0,\infty )$ such that 
for any $\varepsilon\in (0,C\rbrack$, 
$(L,d_{\varepsilon})$ defined by \eqref{L-distance-def} is a metric space that satisfies 
the $(2+2)$-point comparison and the $(4+2)$-point comparison, 
but does not admit a distance-preserving embedding into any $\mathrm{CAT}(0)$ space. 
\end{theorem}

For the definitions of the $(2+2)$-point comparison and the $(4+2)$-point comparison, see \cite[Section 6.2]{AKP}. 
It is known that a metric space satisfies the $(2+2)$-comparison if and only if it satisfies the $\boxtimes$-inequalities. 
Recently, 
the present author \cite{toyoda-ANN} proved the following theorem, 
which states that Lebedeva's $6$-point metric spaces satisfy many $\mathrm{CAT}(0)$ quadratic metric inequalities, including the $\boxtimes$-inequalities 
as special cases. 

\begin{theorem}[\cite{toyoda-ANN}]\label{L-ANN-th}
Suppose $x_{0},x_{1},y_{0},y_{1},z_{0},z_{1}\in\mathbb{R}^3$ 
are distinct six points that satisfy \eqref{lebedeva-condition}. 
Set $L=\{ x_{0},x_{1},y_{0},y_{1},z_{0},z_{1}\}$. 
Then there exists $C\in (0,\infty )$ such that 
for any $\varepsilon\in (0,C\rbrack$, 
$(L,d_{\varepsilon})$ defined by \eqref{L-distance-def} is a metric space that satisfies 
the quadratic metric inequality \eqref{ANN-ineq} 
for any $m$, $n$, $c_1 ,\ldots ,c_m$, $\left( p_{i}^{1}\right)_{i\in\lbrack n\rbrack},\ldots ,\left( p_{i}^{m}\right)_{i\in\lbrack n\rbrack}$, 
$\left( q_{j}^{1}\right)_{j\in\lbrack n\rbrack},\ldots ,\left( q_{j}^{m}\right)_{j\in\lbrack n\rbrack}$, 
$A_1 ,\ldots ,A_m$, $B_1 ,\ldots ,B_m$ as in the statement of Theorem~\ref{ANN-ineq-th}. 
\end{theorem}

\section{Proof of Theorem~\ref{6-pt-ineq-th}}\label{proof-sec}

In this section, we prove Theorem~\ref{6-pt-ineq-th}, Corollary~\ref{6-pt-ANN-coro} and Proposition~\ref{O3-prop}. 
In the proof of Theorem~\ref{6-pt-ineq-th} below, we appropriately combine inequalities discussed in Section \ref{preliminaries-sec} 
so that the resulting $\mathrm{CAT}(0)$ quadratic metric inequality holds with equality for Lebedeva's six points in $\mathbb{R}^3$.

\begin{proof}[Proof of Theorem~\ref{6-pt-ineq-th}]
Let $(X,d_{X})$ be a $\mathrm{CAT}(0)$ space. 
Fix $x_0 ,x_1 ,y_0 ,y_1 ,z_0 ,z_1\in X$ and $a,b,c,s,t\in\lbrack 0,1\rbrack$ with $a\leq s$. 
If $a=0$, then the desired inequality becomes 
\begin{equation*}
s(1-t)t
\left(  (1-b)d_{X}(x_{0},y_{0})^2 +b d_{X}(x_{0},y_{1})^2-(1-b)b d_{X}(y_{0},y_{1})^2 \right)
\geq 0,
\end{equation*}
which holds true by \eqref{quadratic-tri-ineq}. 
If $b=0$, then the desired inequality becomes 
\begin{equation*}
s(1-t)t
\left( (1-a) d_{X}(x_{0},y_{0})^2 +a d_{X}(x_{1},y_{0})^2 -(1-a)a d_{X}(x_{0},x_{1})^2 \right)
\geq 0,
\end{equation*}
which holds true by \eqref{quadratic-tri-ineq}. 
So henceforth we assume that $a>0$ and $b>0$. 

We set 
\begin{equation*}
z=\mathrm{bar}\left( (1-c)\delta_{z_0}+c\delta_{z_1}\right) ,\quad
w=\mathrm{bar}\left( (1-t)\delta_{y_1}+(1-s)t\delta_{x_0}+st\delta_{x_1}\right) .
\end{equation*}
By applying \eqref{double-variance-ineq}, we obtain 
\begin{align}
&d_{X}(w,z)^2 
+(1-t)d_{X}(w,y_{1})^2 +(1-s)t d_{X}(w,x_{0})^2 +st d_{X}(w,x_{1})^2 \nonumber\\
&\hspace{4pt}+(1-c)d_{X}(z,z_{0})^2 +c d_{X}(z,z_{1})^2 \nonumber\\
&\hspace{8pt}\leq 
(1-t)(1-c)d_{X}(y_{1},z_{0})^2 +(1-s)t(1-c)d_{X}(x_{0},z_{0})^2 +st(1-c)d_{X}(x_{1},z_{0})^2 \nonumber\\
&\hspace{12pt}+(1-t)c d_{X}(y_{1},z_{1})^2 +(1-s)tc d_{X}(x_{0},z_{1})^2 +stc d_{X}(x_{1},z_{1})^2 .\label{6-pt-doble-variance-ineq}
\end{align}
We set 
\begin{equation*}
x=\mathrm{bar}\left( (1-s)\delta_{x_0}+s\delta_{x_1}\right) .
\end{equation*}
By applying \eqref{CAT(0)-def-ineq} and \eqref{quadratic-tri-ineq}, we obtain
\begin{align}
(1-t)d_{X}(w,y_{1})^2 &+(1-s)t d_{X}(w,x_{0})^2 +st d_{X}(w,x_{1})^2 \nonumber\\
&\geq
(1-t)d_{X}(w,y_{1})^2 +t\left( d_{X}(w,x)^2 +(1-s)s d_{X}(x_{0},x_{1})^2 \right) \nonumber\\
&=
\left( (1-t)d_{X}(w,y_{1})^2 +t d_{X}(w,x)^2 \right) +(1-s)st d_{X}(x_{0},x_{1})^2 \nonumber\\
&\geq
(1-t)t d_{X}(y_{1},x)^2 +(1-s)st d_{X}(x_{0},x_{1})^2 .\label{6-pt-double-left1-ineq}
\end{align}
By applying \eqref{quadratic-tri-ineq}, we obtain
\begin{equation}
(1-c)d_{X}(z,z_{0})^2 +c d_{X}(z,z_{1})^2
\geq
(1-c)c d_{X}(z_{0},z_{1})^2 .\label{6-pt-double-left2-ineq}
\end{equation}
By \eqref{6-pt-doble-variance-ineq}, \eqref{6-pt-double-left1-ineq} and \eqref{6-pt-double-left2-ineq}, we have 
\begin{multline}\label{6-pt-moto-ineq}
d_{X}(w,z)^2 
+(1-t)t d_{X}(y_{1},x)^2 +(1-s)st d_{X}(x_{0},x_{1})^2 +(1-c)c d_{X}(z_{0},z_{1})^2 \\
\leq
(1-t)(1-c)d_{X}(y_{1},z_{0})^2 +(1-s)t(1-c)d_{X}(x_{0},z_{0})^2 +st(1-c)d_{X}(x_{1},z_{0})^2 \\\
+(1-t)c d_{X}(y_{1},z_{1})^2 +(1-s)tc d_{X}(x_{0},z_{1})^2 +stc d_{X}(x_{1},z_{1})^2. 
\end{multline}
We set 
\begin{equation*}
x'=\mathrm{bar}\left( (1-a)\delta_{x_0}+a\delta_{x_1}\right) .
\end{equation*}
We have 
\begin{equation}\label{6-pt-b-tri-ineq}
b d_{X}(y_{1},x' )^2
\geq
(1-b)b d_{X}(y_{0},y_{1})^2 -(1-b)d_{X}(y_{0},x' )^2
\end{equation}
by \eqref{quadratic-tri-ineq}, and 
\begin{equation}\label{6-pt-a-thin-ineq}
d_{X}(y_{0},x' )^2
\leq
(1-a)d_{X}(y_{0},x_{0})^2 +a d_{X}(y_{0},x_{1})^2
-(1-a)a d_{X}(x_{0},x_{1})^2
\end{equation}
by \eqref{CAT(0)-def-ineq}. 
Since we are assuming that $b>0$, combining \eqref{6-pt-b-tri-ineq} and \eqref{6-pt-a-thin-ineq}, we obtain 
\begin{multline}\label{6-pt-thick-xdash-ineq}
d_{X}(y_{1},x' )^2 
\geq
(1-b)d_{X}(y_{0},y_{1})^2 
-\frac{(1-a)(1-b)}{b}d_{X}(y_{0},x_{0})^2 \\
-\frac{a(1-b)}{b}d_{X}(y_{0},x_{1})^2
+\frac{(1-a)a(1-b)}{b}d_{X}(x_{0},x_{1})^2 .
\end{multline}
Since we are assuming that $a>0$, it follows from Proposition~\ref{thick-prop} that 
\begin{equation}\label{6-pt-thick-prop-ineq}
d_{X}(y_{1},x)^2
\geq
\frac{s}{a}d_{X}(y_{1},x' )^2 -\frac{s-a}{a}d_{X}(y_{1},x_{0})^2 +s(s-a)d_{X}(x_{0},x_{1})^2 .
\end{equation}
Combining \eqref{6-pt-thick-xdash-ineq} and \eqref{6-pt-thick-prop-ineq}, we obtain 
\begin{multline}\label{6-pt-thick-ineq}
d_{X}(y_{1},x)^2
\geq
\frac{s\left( (1-a)(1-b)+(s-a)b\right)}{b}d_{X}(x_{0},x_{1})^2
+\frac{s(1-b)}{a}d_{X}(y_{0},y_{1})^2 \\
-\frac{s(1-a)(1-b)}{ab}d_{X}(y_{0},x_{0})^2 
-\frac{s(1-b)}{b}d_{X}(y_{0},x_{1})^2 
-\frac{s-a}{a}d_{X}(y_{1},x_{0})^2 .
\end{multline}
Since we are assuming that $a>0$ and $b>0$, combining \eqref{6-pt-moto-ineq} and \eqref{6-pt-thick-ineq}, we obtain 
\begin{align}
&ab\, d_{X}(w,z)^2 
+sta \big( (1-t)(1-a)+(1-s)tb\big) \, d_{X}(x_0 ,x_1 )^2 \nonumber\\
&\hspace{4pt}+s(1-t)t(1-b)b \, d_{X}(y_0 ,y_1 )^2 
+ab(1-c)c \, d_{X}(z_0 ,z_1 )^2 \nonumber\\
&\hspace{8pt}\leq 
(1-s)tab(1-c)\, d_{X}(x_0 ,z_0 )^2
+stab(1-c)\, d_{X}(x_1 ,z_0 )^2 
+(1-t)ab(1-c) \, d_{X}(y_1 ,z_0 )^2 \nonumber\\
&\hspace{12pt}+(1-s)tabc \, d_{X}(x_0 ,z_1 )^2 
+stabc \, d_{X}(x_1 ,z_1 )^2 
+(1-t)abc\, d_{X}(y_1 ,z_1 )^2 \nonumber\\
&\hspace{16pt}+s(1-t)t(1-a)(1-b) \, d_{X}(x_0 ,y_0 )^2
+s(1-t)ta(1-b) \, d_{X}(x_1 ,y_0 )^2 \nonumber\\ 
&\hspace{20pt}+(1-t)t(s-a)b \, d_{X}(x_0 ,y_1 )^2 ,
\label{6-point-deseired-ineq}
\end{align}
which implies the desired inequality \eqref{6-pt-ineq}.

To prove the moreover part of the theorem, we assume that 
$a,b,c,s,t\in(0,1)$ and $a<s$ from now on. 
Furthermore, we assume that $X$ is the three dimensional Euclidean space $\mathbb{R}^3$, 
and the points $x_0 ,x_1 ,y_0 ,y_1 ,z_0 ,z_1 \in X$ satisfy \eqref{lebedeva-condition} and 
\begin{equation*}
(1-t)y_1 +t\left( (1-s) x_{0}+s x_{1}\right)
=
(1-c)z_{0}+c z_{1},\quad
(1-a)x_{0}+a x_{1}
=
(1-b) y_{0}+b y_{1}.
\end{equation*}
Such $x_0 ,x_1 ,y_0 ,y_1 ,z_0 ,z_1$ exist in $\mathbb{R}^3$ clearly 
(see \textsc{Figure}~\ref{R3-fig}). 
Then we have 
\begin{align*}
w
&=
(1-t)y_{1}+t \left( (1-s)x_{0}+s x_{1}\right)
=
(1-c)z_{0}+cz_{1}
=
z,\\
x
&=
(1-s)x_{0}+sx_{1},\quad
x'
=
(1-a)x_{0}+a x_{1}
=
(1-b)y_{0}+b y_{1} .
\end{align*}
It follows that all inequalities in \eqref{6-pt-doble-variance-ineq}--\eqref{6-point-deseired-ineq} hold with equality. 
Therefore, equality holds in \eqref{6-pt-ineq} as well since we have $w=z$.

\begin{figure}[htbp]
\centering
\begin{tikzpicture}[scale=0.6]
\draw [opacity=0.3,fill=gray] (0,0) -- (3,5) -- (14,5) -- (11,0) -- (0,0);
\draw (3,1) -- (11,4);
\draw (10,1) -- (4,4);
\draw (9,6) -- (7.9,1.6);
\draw (10,1) -- (223/37,79/37);
\draw (7.9,1.6)[dotted] -- (7.5,0);
\draw (7.5,0) -- (7,-2);
\draw [fill] (3,1) circle [radius=0.07];
\draw [fill] (11,4) circle [radius=0.07];
\draw [fill] (9,6) circle [radius=0.07];
\draw [fill] (7.9,1.6) circle [radius=0.07];
\draw [fill] (7,-2) circle [radius=0.07];
\draw [fill] (10,1) circle [radius=0.07];
\draw [fill] (4,4) circle [radius=0.07];
\draw [fill] (223/37,79/37) circle [radius=0.07];
\draw [fill] (7,2.5) circle [radius=0.07];
\node [right] at (9,6) {$z_1$};
\node [right] at (7,-2) {$z_0$};
\node [right] at (11,4) {$x_0$};
\node [left] at (3,1) {$x_1$};
\node [right] at (10,1) {$y_1$};
\node [left] at (4,4) {$y_0$};
\node [below left] at (7.9,1.6) {$w=z$};
\node [above] at (7,2.5) {$x'$};
\node [above] at (223/37,79/37) {$x$};
\end{tikzpicture}
\caption{The points $x_0 ,x_1 ,y_0 ,y_1 ,z_0 ,z_1$ in $\mathbb{R}^3$. }
\label{R3-fig}
\end{figure}

Let $L=\{ x_0 ,x_1 ,y_0 ,y_1 ,z_0 ,z_1 \}$, and define $d_{\varepsilon}\colon L\times L\to\lbrack 0,\infty )$ by 
\eqref{L-distance-def} for each $\varepsilon\in\lbrack 0,\infty )$. 
We have observed that equality holds in \eqref{6-pt-ineq} if we choose $x_0 ,x_1 ,y_0 ,y_1 ,z_0 ,z_1$ as above, 
and set $d_{X}=d_{0}$. 
Since the coefficient of $d_X (z_{0},z_{1})^2$ in \eqref{6-pt-ineq} is positive, 
it follows that for any $\varepsilon\in (0,\infty )$, 
\eqref{6-pt-ineq} does not hold true if we choose $x_0 ,x_1 ,y_0 ,y_1 ,z_0 ,z_1$ as above, and set $d_{X}= d_{\varepsilon}$. 
On the other hand, for a sufficiently small $\varepsilon\in (0,\infty )$, $(L,d_{\varepsilon})$ satisfies the $\boxtimes$-inequalities by Theorem~\ref{Lebedeva-th}, 
and therefore any $5$-point subset of $(L,d_{\varepsilon})$ admits 
a distance-preserving embedding into a $\mathrm{CAT}(0)$ space by Theorem~\ref{five-th}, which completes the proof. 
 \end{proof}

\begin{proof}[Proof of Corollary~\ref{6-pt-ANN-coro}]
As we observed in the proof of Theorem~\ref{6-pt-ineq-th}, 
for any $a,b,c,s,t\in(0,1)$ with $a<s$, 
there exists a Lebedeva's family of spaces $(L,d_{\varepsilon})$, $\varepsilon >0$ such that 
each $(L,d_{\varepsilon})$ does not satisfy the inequality \eqref{6-pt-ineq} for the parameters 
$a,b,c,s,t$. 
On the other hand, by Theorem~\ref{L-ANN-th}, for a sufficiently small $\varepsilon >0$, 
$(L,d_{\varepsilon})$ satisfies all inequalities of the form \eqref{ANN-ineq}, 
which proves the corollary. 
\end{proof}

\begin{proof}[Proof of Proposition~\ref{O3-prop}]
Fix $x_0 ,x_1 ,y_0 ,y_1 ,z_0 ,z_1 \in X$. 
Since $X$ satisfies the $O_3$-comparison, there exist points 
$\tilde{x}_0 ,\tilde{x}_1 ,\tilde{y}_0 ,\tilde{y}_1 ,\tilde{z}_0 ,\tilde{z}_1$ in a Hilbert space $\mathcal{H}$ 
such that 
\begin{equation}\label{abc-ineqs}
\|\tilde{p}_{0}-\tilde{p}_{1}\|\geq d_X (p_{0} ,p_{1}),\quad
\|\tilde{q}_{i}-\tilde{r}_{j}\|\leq d_X (q_{i} ,r_{j})
\end{equation}
for any $i,j\in\{ 0,1\}$ and any $p,q,r\in\{ x,y,z\}$ with $q\neq r$. 
Because $\mathcal{H}$ is a $\mathrm{CAT}(0)$ space, 
the points $\tilde{x}_0 ,\tilde{x}_1 ,\tilde{y}_0 ,\tilde{y}_1 ,\tilde{z}_0 ,\tilde{z}_1$ satisfy the inequality of the form 
\eqref{6-pt-ineq} by Theorem~\ref{6-pt-ineq-th}. 
Combining this fact with \eqref{abc-ineqs}, we can observe that 
the points $x_0 ,x_1 ,y_0 ,y_1 ,z_0 ,z_1 \in X$ satisfy the inequality \eqref{6-pt-ineq}. 
\end{proof}

\bigskip
\begin{acknowledgement}
This work was supported in part by JSPS KAKENHI Grant Number JP21K03254. 
The author thanks Tadashi Fujioka, Takefumi Kondo and Takayuki Morisawa for helpful discussions. 
The author also thanks the anonymous referee for valuable comments.
\end{acknowledgement}


\begin{thebibliography}{00}

\bibitem{AKP}
S.~B. Alexander, V. Kapovitch and A. Petrunin, Alexandrov meets Kirszbraun, in {\it Proceedings of the G\"okova Geometry-Topology Conference 2010}, 88--109, Int. Press, Somerville, MA, 2011; updated version: \href{https://arxiv.org/abs/1012.5636}{arXiv:1012.5636v2}, 2017.


\bibitem{AKP-book}
S.~B. Alexander, V. Kapovitch and A. Petrunin, {\it Alexandrov geometry---foundations}, Graduate Studies in Mathematics, 236, Amer. Math. Soc., Providence, RI, [2024] \copyright 2024.

\bibitem{ANN}
A. Andoni, A. Naor and O. Neiman, Snowflake universality of Wasserstein spaces, Ann. Sci. \'Ec. Norm. Sup\'er. (4) {\bf 51} (2018), no.~3, 657--700.

\bibitem{Bac}
M. Ba\v c\'ak, Old and new challenges in Hadamard spaces, Jpn. J. Math. {\bf 18} (2023), no.~2, 115--168.

\bibitem{BH}
M.~R. Bridson and A. Haefliger, {\it Metric spaces of non-positive curvature}, Grundlehren der mathematischen Wissenschaften, 319, Springer, Berlin, 1999.

\bibitem{BBI}
D.~Y. Burago, Y.~D. Burago and S.~V. Ivanov, {\it A course in metric geometry}, Graduate Studies in Mathematics, 33, Amer. Math. Soc., Providence, RI, 2001.

\bibitem{Gr1}
M. Gromov, {\it Metric structures for Riemannian and non-Riemannian spaces}, translated from the French by Sean Michael Bates, 
Progress in Mathematics, 152, Birkh\"auser Boston, Boston, MA, 1999.
 

\bibitem{Gr2} 
M. Gromov, J. Math. Sci. (N.Y.) {\bf 119} (2004), no.~2, 178--200; translated from Zap. Nauchn. Sem. S.-Peterburg. Otdel. Mat. Inst. Steklov. (POMI) {\bf 280} (2001), 100--140, 299--300.

\bibitem{Kl}
B. Kleiner, The local structure of length spaces with curvature bounded above, Math. Z. {\bf 231} (1999), no.~3, 409--456.

  
\bibitem{LP-to}
N.~D. Lebedeva and A. Petrunin, 5-point $\rm CAT(0)$ spaces after Tetsu Toyoda, Anal. Geom. Metr. Spaces {\bf 9} (2021), no.~1, 160--166.

\bibitem{LP-tree}
N.~D. Lebedeva and A. Petrunin, Trees meet octahedron comparison, J. Topol. Anal., online ready, 
https://doi.org/10.1142/S1793525323500231, 2023.

\bibitem{LP-graph}
N.~D. Lebedeva and A. Petrunin, Graph comparison meets Alexandrov, Sib. Math. J. {\bf 64} (2023), no.~3, 624--628.

\bibitem{LPZ}
N.~D. Lebedeva, A. Petrunin and V. Zolotov, Bipolar comparison, Geom. Funct. Anal. {\bf 29} (2019), no.~1, 258--282.

\bibitem{R}
Y.~G. Reshetnyak, Non-expansive maps in a space of curvature no greater than $K$, Sibirsk. Mat. \v Z. {\bf 9} (1968), 918--927.

\bibitem{St}
K.-T. Sturm, Probability measures on metric spaces of nonpositive curvature, in {\it Heat kernels and analysis on manifolds, graphs, and metric spaces (Paris, 2002)}, 357--390, Contemp. Math., 338, Amer. Math. Soc., Providence, RI, 2003.

\bibitem{toyoda-five}
T. Toyoda, An intrinsic characterization of five points in a ${\rm CAT}(0)$ space, Anal. Geom. Metr. Spaces {\bf 8} (2020), no.~1, 114--165.

\bibitem{toyoda-cycle}
T. Toyoda, A non-geodesic analogue of Reshetnyak's majorization theorem, Anal. Geom. Metr. Spaces {\bf 11} (2023), no.~1, Paper No. 20220151, 22 pp.

\bibitem{toyoda-ANN}
T. Toyoda, The Andoni-Naor-Neiman inequalities and isometric embeddability into a $\mathrm{CAT}(0)$ space, \href{https://arxiv.org/abs/2404.13871}{arXiv:2404.13871}, 2024.
\end{thebibliography}
\end{document}